\documentclass[11pt, letter, reqno]{amsart}
\usepackage{graphicx,psfrag}
\usepackage[psamsfonts]{amssymb}
\usepackage{pdfsync}
\usepackage{amscd}
\usepackage{amsmath}
\usepackage{mathrsfs}
\usepackage{stmaryrd}
\usepackage{comment}
\usepackage{xfrac}
\usepackage[margin=1in]{geometry}
\usepackage{comment}
\usepackage{enumitem}
\usepackage{stmaryrd}
\theoremstyle{plain}
    \newtheorem{theorem}{Theorem}[section]

    \newtheorem{proposition}[theorem]{Proposition}
    \newtheorem{conjecture}[theorem]{Conjecture}
    
\theoremstyle{definition}
    \newtheorem{definition}[theorem]{Definition}

    \newtheorem*{thank}{Acknowledgements}
\theoremstyle{remark}

\numberwithin{equation}{section}

\newcommand{\ZZ}{\mathbb{Z}}

\newcommand{\real}{\mathbb{R}}
\newcommand{\complex}{\mathbb{C}}

\newcommand{\OO}{\mathcal{O}}

\newcommand{\E}{\mathcal{E}}

\newcommand{\M}{\mathcal{M}}
\newcommand{\N}{\mathbb{N}}

\newcommand{\Lie}{\operatorname{Lie}}
\newcommand{\Lieg}{\mathfrak{g}}
\newcommand{\Liek}{\mathfrak{k}}
\newcommand{\Lieh}{\mathfrak{h}}
\newcommand{\Lien}{\mathfrak{n}}
\newcommand{\Lieb}{\mathfrak{b}}

\newcommand{\Liec}{\mathfrak{c}}
\newcommand{\Lies}{\mathfrak{s}}
\newcommand{\Liest}{\mathfrak{st}}
\newcommand{\Lieu}{\mathfrak{u}}
\newcommand{\Liet}{\mathfrak{t}}
\newcommand{\Liea}{\mathfrak{a}}

\newcommand{\LLieg}{\mathfrak{g}^\circ}
\newcommand{\LLieb}{\mathfrak{b}^\circ}
\newcommand{\LLien}{\mathfrak{n}^\circ}
\newcommand{\LLiek}{\mathfrak{k}^\circ}
\newcommand{\LLies}{\mathfrak{s}^\circ}
\newcommand{\LLieh}{\mathfrak{h}^\circ}
\newcommand{\LLiet}{\mathfrak{t}^\circ}
\newcommand{\LLiea}{\mathfrak{a}^\circ}
\newcommand{\UU}{\mathcal{U} \Lieg^\circ}
\newcommand{\Dh}{\mathcal{D}_\mathfrak{h}}
\newcommand{\DQ}{\D^{\lambda}_{X \gets Q}}
\newcommand{\DhQ}{\mathcal{D}_{\Lieh, Q}}
\newcommand{\DlQ}{\mathcal{D}_{Q}^{\lambda_t}}
\newcommand{\Dlambda}{\mathcal{D}_\lambda}
\newcommand{\MM}{\mathcal{M}}

\newcommand{\Ug}{\mathcal{U}\Lieg}
\newcommand{\Uh}{\mathcal{U}\Lieh}
\newcommand{\Zg}{{\mathcal{Z}\Lieg}}
\newcommand{\U}{\mathcal{U}}

\newcommand{\D}{\mathcal{D}}

\newcommand{\LL}{\mathcal{L}}
\newcommand{\I}{\mathcal{I}}

\newcommand{\Ind}{\operatorname{Ind}}

\newcommand{\Prin}{\mathcal{P}}
\newcommand{\supp}{\operatorname{Supp}}
\newcommand{\Rees}{\operatorname{Rees}}

\newcommand{\spn}{\operatorname{span}}
\newcommand{\SLR}{SL(2, \real)}
\newcommand{\SLC}{SL(2, \complex)}
\newcommand{\SU}{SU(1,1)}

\newcommand{\slc}{\mathfrak{sl}(2, \complex)}

\newcommand{\PP}{\mathbb{P}^1}
\newcommand{\dz}{\partial_{z}}
\newcommand{\dzz}{\partial_{\zeta}}

\newcommand{\Img}{\mathfrak{Im}}

\begin{document}

\title{Mackey analogy via $\D$-modules in the example of $SL(2,\real)$}
\author{Qijun Tan, Yijun Yao, Shilin Yu}

%\date{\today}

%\abovedisplayskip=4pt
%\belowdisplayskip=4pt

\setlength{\abovedisplayskip}{7pt}
\setlength{\belowdisplayskip}{7pt}

\maketitle

\begin{abstract}
  A conjecture by Mackey and Higson claims that there is close relationship between irreducible representations of a real reductive group and those of its Cartan motion group. The case of irreducible tempered unitary representations has been verified recently by Afgoustidis. We study the admissible representations of $\SLR$ by considering families of $\D$-modules over its flag varieties. We make a conjecture which gives a geometric understanding of the Makcey-Higson bijection in the general case.
\end{abstract}

\section{Introduction}\label{sec:intro}

Inspired by the concept from physics of the contraction of a Lie group to a Lie subgroup (\cite{InonuWigner}), Mackey suggested in 1975 (\cite{Mackey}) that there should be a correspondence between ``almost" all the irreducible unitary representations of a noncompact semisimple group $G_\real$ and the irreducible unitary representations of its contraction to a maximal compact subgroup $K_\real$. The contraction group is defined to be the group
\[ G_{\real,0} :=K_\real \ltimes \Lieg_\real / \Liek_\real, \]
where $\Lieg_\real = \Lie(G_\real)$ and $\Liek_\real = \Lie(K_\real)$ are the corresponding Lie algebras and $\Lieg_\real / \Liek_\real$ is regarded as an abelian group with the usual addition of vectors. The group $G_{\real,0}$ is called the \emph{Cartan motion group} of $G_\real$. It is a surprising analogy since the algebraic structures of the groups $G_\real$ and $G_{\real,0}$ are quite different. The representation theory of the semisimple group $G_\real$ is rather complicated and even decades after Mackey, the problem of finding an effective discription of the unitary dual $\widehat{G_\real}$ is not fully solved yet. On the other hand, Mackey himself developed a full theory of representations of semidirect product groups like $G_{\real,0}$, so the unitary dual of $G_{\real,0}$ is much easier to describe. 

Later Connes pointed out that there is a connection between the Mackey analogy and the Connes-Kasparov conjecture in K-theory of $C^*$-algebras (\cite{BCH}), which suggests that the reduced dual, or equivalently, the tempered dual of $G_\real$ should correspond to the unitary dual of $G_{\real,0}$, at least K-theoretically. Following Connes' insight, Higson suggested that the correspondence  ought to be a set theoretical bijection. In other words, Mackey analogy can be regarded as a stronger version of the Connes-Kasparov conjecture. In his paper \cite{HigsonMackeyKtheory}, Higson examined the case where $G_\real$ is a connected complex semisimple group (regarded as a real group) and showed that there is a natural bijection between the reduced duals of $G_\real$ and $G_{\real,0}$, with the already known classification of irreducible tempered representations on both sides in hand. Later in \cite{HigsonMackey}, he strengthened this result by showing that there is even a natural bijection between the admissible duals of $G_\real$ and $G_{\real,0}$ when $G_\real$ is a complex group.

Whether there is a Mackey bijection between tempered duals or even admissible duals when $G_\real$ is a real group has remained unsolved for a long time. Recently in \cite{Afgoustidis}, Afgoustidis has established a very clean and natural bijection between the tempered dual of $G_\real$ and the unitary dual of $G_{\real,0}$ using the Knapp-Zuckerman classification of tempered irreducible representations of $G_\real$ (\cite{KnappZuckerman_I}, \cite{KnappZuckerman_II}). This bijection is in particular an extension of Vogan's bijection between irreducible tempered representation of $G_\real$ with real infinitesimal characters and their unique minimal $K$-types (\cite{Vogan}). Moreover, Afgoustidis used his Mackey-Higson bijection to give a new proof of the Connes-Kasparov isomorphism for real reductive Lie groups (\cite{Afgoustidis_CK}).

Afgoustidis also studied the Mackey analogy at the level of representation spaces by writing down explict contractions from representations of $G_\real$ to that of $G_{\real,0}$ in the case of spherical principal series representations, discrete series and limit of discrete series representations. However, it is not clear yet if there is a general way to construct such contractions for all tempered representations, even for those with real infinitesimal characters. One difficulty is that the bijection behaves poorly at the level of representation spaces. For instance, while there are unitary unitary irreducible representations of $G_{\real,0}$ whose underlying vector spaces are of finite dimensions (on which the $\Lieg_\real / \Liek_\real$ part of $G_{\real,0}$ acts trivially), all nontrivial unitary representations of $G_\real$ are infinite-dimensional. 

We propose to study the Mackey analogy from the perspective of $\D$-modules. In particular, we regard representations as $\D$-modules over the flag variety via the Beilinson-Bernstein localization theorem (\cite{BeilinsonBernstein}) and show how to deform them to get representations of the Cartan motion group.  We expect that our construction works more generally for admissible representations of $G_\real$. We will state our main conjecture in Conjecture \ref{conj:Mackey} at the end of \S~\ref{subsec:deform}. We will compute the example of $\SLR$ in \S~\ref{sec:sl2} to illustrate our conjecture.

\begin{thank}
  We would like to thank Nigel Higson for introducing us to the subject and numerous discussions.
\end{thank}

\section{Mackey-Higson correspondence}\label{sec:Mackey}

\subsection{Basics of $\D$-modules}\label{subsec:Dmod}

We recall the construction of twisted $\D$-modules on the flag variety from \cite{Milicic}. Let $X$ be the flag variety of $G$, which is the variety of all Borel subalgebras $\Lieb$ in $\Lieg$. Let $\LLieg=\OO_X \otimes_\complex \Lieg$ be the sheaf of local sections of the trivial bundle $X \times \Lieg$. Let $\LLieb$ be the vector bundle on $X$ whose fiber $\Lieb_x$ at any point $x$ of $X$ is the Borel subalgebra $\Lieb \subset \Lieg$ corresponding to $x$. Similarly, let $\LLien$ be the vector bundle whose fiber $\Lien_x$ is the nilpotent ideal $\Lien_x = [\Lieb_x, \Lieb_x]$ of the corresponding Borel subalgebra $\Lieb$. $\LLieb$ and $\LLien$ can be considered subsheaves of $\LLieg$. The sheaf $\LLieg$ has a natural structure of Lie algebroid: the differential of the action of $G$ on $X$ defines a natural map from $\Lieg$ to the tangent bundle $TX$ of $X$ and hence induces an anchor map $\tau: \LLieg \to TX$. The Lie structure on $\Lieg$ is given by
  \[  [f \otimes \xi, g \otimes \eta] = f \tau(\xi) g \otimes \eta - g \tau(\eta)f \otimes \xi + fg \otimes [\xi, \eta] \]
for any $f, g \in \OO_X$ and $\xi, \eta \in \Lieg$. The kernel of $\tau$ is exactly $\LLieb$, so $\LLieb$ and $\LLien$ are sheaves of Lie ideals in $\LLieg$.

We then form the universal enveloping algebra of the Lie algebroid $\LLieg$, which is the sheaf $\UU=\OO_X \otimes_\complex \Ug$ of associative algebras with the multiplication defined by
  \[  (f \otimes \xi) (g \otimes \eta) = f \tau(\xi) g \otimes \eta + fg \otimes \xi \eta \]
for $f,g \in \OO_X$ and $\xi \in \Lieg$, $\eta \in \Ug$. The sheaf of left ideals $\UU \LLien $ generated by $\LLien$ in $\UU$ is a sheaf of two-sided ideals in $\UU$, hence the quotient $\Dh = \UU /  \UU \LLien$ is a sheaf of associative algebras on $X$.

The natural morphism from $\LLieg$ to $\Dh$ induces an inclusion of $\LLieh=\LLieb/\LLien$ into $\Dh$. The sheaf $\LLieh$ turns out to be a trivial vector bundle and its global sections over $X$ is the \emph{abstract Cartan algebra} $\Lieh$ of $\Lieg$, which is independent of the choice of Borel subalgebra: for any Cartan subalgebra $\Liec$ of $\Lieg$ and any Borel subalgebra $\Lieb_x$ containing $\Liec$, the composition $\Liec \to \Lieb_x \to \Lieb_x / \Lien_x \simeq \Gamma(X,\LLieh) = \Lieh$ is a canonical isomorphism which is independent of choice of $x \in X$. Moreover, we also have abstract root system $\Sigma$ and positive root system $\Sigma^+$ in $\Lieh^*$ which consists of the set of roots of $\Lieh$ in $\Lieg/\Lieb_x$, as well as the abstract Weyl group $W$. The natural action of $G$ on $\LLieh$ is trivial and embedding $\LLieh \hookrightarrow \Dh$ identifies the universal enveloping algebra $\Uh = S \Lieh$ of the abelian Lie algebra $\Lieh$ with the $G$-invariant part of $\Gamma(X,\Dh)$. On the other hand, the center $\Zg$ of $\Ug$ is also naturally contained in $\Gamma(X,\Dh)^G$ and the induced map $\gamma: \Zg \to \Uh$ is the well-known Harish-Chandra homomorphism, which identifies $\Zg$ with the $W$-invariant of $\Uh$, where the action of $W$ on $\Lieh^*$ is the usual one twisted by the half sum $\rho$ of positive roots:
  \[ w . \lambda = w (\lambda - \rho) + \rho. \]
We then have $\Gamma(X,\Dh) \simeq \Ug \otimes_\Zg \Uh$.

Any $\lambda \in \Lieh^*$ determines a homomorphism $\Uh$ to $\complex$. Let $I_\lambda$ be the kernel of the homomorphism $\Uh \to \complex$ determined by $\lambda - \rho$. Then $\gamma^{-1}(I_\lambda)$ is a maximal ideal in $\Zg$ and $\gamma^{-1}(I_\lambda) = \gamma^{-1}(I_\mu)$ if and only if $w \cdot \lambda = \mu$ for some $w \in W$ (where we use the usual $W$-action on $\Lieh$). Thus we can denote the kernel by $J_\chi = \gamma^{-1}(I_\lambda)$ where $\chi = W \cdot \lambda$ is the $W$-orbit of $\lambda$ in $\Lieh^*$. We denote the corresponding infinitesimal character by $\chi_\lambda: \Zg \to \complex$. The sheaf $I_\lambda \Dh$ is a subsheaf of two-sided ideals in $\Dh$, therefore $\Dlambda = \Dh / I_\lambda \Dh$ is a sheaf of associative algebras. We have a canonical algebraic isomorphism between the algebra $\mathcal{U}_\chi = \Ug / J_\chi \Ug = \Ug \otimes_\Zg \complex_{\lambda - \rho}$ and $\Gamma(X,\Dh)$.

Let $\MM(\mathcal{U}_\chi)$ be the abelian category of $\mathcal{U}_\chi$-modules. It is the same as the category of $\Ug$-modules with infinitesimal characters determined by $\chi$. We also have the abelian category $\MM_{qc}(\Dlambda)$ of quasi-coherent $\Dlambda$-modules over $X$.

The Beilinson-Bernstein localization theorem (\cite{BeilinsonBernstein}) says that the global section functor $\Gamma: \MM_{qc}(\Dlambda) \to \MM(\mathcal{U}_\chi)$ is an equivalence of abelian categories if $\lambda$ is dominant and regular. The inverse functor is $\Delta_\lambda: \MM(\mathcal{U}_\chi) \to \MM_{qc}(\Dlambda)$, given by $\Delta_\lambda(V) = \Dlambda \otimes_{\mathcal{U}_\chi} V$, for any $V \in \MM(\mathcal{U}_\chi)$.

Therefore we can localize a $(\Lieg, K)$-module with infinitesimal character $\chi$ to get a $K$-equivariant $\Dlambda$-module on $X$ by choosing a $\lambda \in \Lieh^*$ in the $W$-orbit $\chi$. To get irreducible $(\Lieg,K)$-modules, the standard way is to take a $K$-orbit $Q$ in $X$ together with an irreducible $K$-homogeneous connection $\phi$ on $Q$, which satisfies certain compatibility condition with $\lambda$, and then push forward $\phi$ to a $\Dlambda$-module $\I(Q,\phi)$ on $X$, called the standard Harish-Chandra sheaf attached to $(Q,\phi)$. More precisely, we define the transfer bimodule by
  \begin{equation}\label{eq:transfer}
    \DQ := i^{-1}(\D_\lambda) \otimes_{i^{-1} \OO_{X}} \omega_{Q/X},
  \end{equation}
where $\omega_{Q/X} = \omega_X^{-1} \otimes_{i^{-1} \OO_X} \omega_Q$ is the relative canonical bundle of $Q$ in $X$. Note that the tensor in the definition is with respect to the right $\OO_X$-module structure of $\Dlambda$. The sheaf $\D^{\lambda}_{X \gets Q}$ is a left $i^{-1} \D_\lambda$-module. The restriction of $\LLiek$ to the $K$-orbit $Q$ is still a Lie algebroid, even though $\LLieg$ no longer is. The universal enveloping algebra $\U (\LLiek)$ acts on $\DQ$ from the right, where $\LLiek$ stands for the restriction $i^*\LLiek$ of $\LLiek$ to $Q$.  The same convention applies to other sheaves.

Now let $\phi$ be an irreducible $K$-homogeneous connection on $Q$. Let $x \in Q$ and $T_x(\phi)$ be the geometric fiber of $\phi$ at $x$. Then $T_x(\phi)$ is an irreducible finite dimensional representation of the stabilizer $St_x$ of $x$ in $K$. The connection $\phi$ is completely determined by this representation of $St_x$ on $T_x(\phi)$ since $\phi$ is $K$-homogeneous. Let $\Liec$ be a $\theta$-stable Cartan subalgebra in the Borel subalgebra $\Lieb_x$. The Lie algebra $\Liest_x = \Liek \cap \Lieb_x$ of $St_x$ is the semidirect product of the \emph{toroidal part} $\Liet_Q=\Liek \cap \Liec$ with the nilpotent radical $\Lieu_x = \Liek \cap \Lien_x$ of $\Liest_x$. Let $U_x$ be the unipotent subgroup of $K$ corresponding to $\Lieu_x$. It is unipotent radical of $St_x$. Let $T$ be the Levi factor of $St_x$ with Lie algebra $\Liet_Q$, then $St_x$ is the semidirect product of $T$ with $U_x$. The representation of $St_x$ in $T_x(\phi)$ is trivial on $U_x$, so it can be viewed as a representation of $T$. We say that $\phi$ is \emph{compatible} with $\lambda-\rho$ if the differential of this representation decomposes into a direct sum of a finite number of copies of the one dimensional representation determined by the restriction of $\lambda-\rho$ (specialized to $\Liec$) to $\Liet_Q$ (note that $T$ might be nonabelian in general). 

Now given a $K$-homogeneous connection $\phi$ on $Q$, we can push forward $\phi$ from $Q$ to $X$ to get
  \begin{equation}
    \I(Q,\phi) = i_*(\DQ \otimes_{\U\LLiek} \phi).
  \end{equation}
Under certain condtions, $I(Q,\phi)$ is irreducible and produces an irreducible $(\Lieg,K)$-module, otherwise it contains a unique irreducible subsheaf of $\Dlambda$-modules, denoted by $\LL(Q,\phi)$. With certain assumptions on $\lambda$ and $Q$, the cohomologies of such $\LL(Q,\phi)$ give a geometric classification of the admissible representations of $G_\real$ (\cite{Localization}). The classification of irreducible tempered representations in terms of standard Harish-Chandra sheaves was done in \cite{Chang} and \cite{Mirkovic}.

\subsection{Deformation of $\D$-modules}\label{subsec:deform}
We will describe how to realize $(\Lieg_0,K)$-modules as geometric objects on the flag variety $X$ and how twisted $\D$-modules deform to them. We will not give proofs about validity of the constructions, which will appear elsewhere. Instead, we make the Conjecture \ref{conj:Mackey} and illustrate it by considering the case of $\SLR$ in \S \ref{sec:sl2}.  

First of all, the Lie algebras $\Lieg_0$ and $\Lieg$ fit into a continuous family of Lie algebras $\Lieg_t$, $t \in \complex$, with the fiber at $t=0$ being the Lie algebra $\Lieg_0$ and other fibers $\Lieg_t$, $t \neq 0$, isomorphic to $\Lieg$. A convenient way to describe it is as follows: take the trivial vector bundle $\complex \times \Lieg$ and regard it as a sheaf of $\OO_\complex$-modules over the affine line $\complex$. It is a sheaf of Lie algebras over $\complex$ and its module of global sections is $\Lieg[t]=\Lieg \otimes_\complex \complex[t]$, where $t$ is the coordinate function of $\complex$. Then we can think of $\Lieg_t$ as the subsheaf in $\complex \times \Lieg$ of germs sections which take values in $\Liek \subset \Lieg$ at $0 \in \complex$. In other words, $\Lieg_t = \Liek[t] \oplus t \Lies[t]$. This is a sheaf of Lie subalgebras.

We can extend this construction to the flag variety $X$ and form the trivial vector bundle $\LLieg[t] = X \times \complex \times \Lieg$ and its subsheaf $\LLieg_t=\LLiek [t] \oplus t\LLies[t]$ of germs of local sections which take values in the trivial vector bundle $\LLiek = X \times \Liek$ over $X \times \{ 0 \}$. Both are Lie algebroids over $X \times \complex$ whose anchor maps are induced by the group actions on $X$ and take values in $TX[t]=TX \otimes_\complex \complex[t]$. The sheaf $\LLieg_t$ is a subsheaf of Lie subalgebras of $\LLieg[t]$. Global sections of $\LLieg_t$ is exactly $\Lieg_t$. We form the sheaf of universal enveloping algebras $\U (\LLieg[t])$ generated by $\OO_X$ and $\LLieg[t]$. We also have the subsheaf of subalgebras $\U(\LLieg_t)$ generated by $\OO_X$ and $\LLieg_t$. 

Now assume a $K$-orbit $Q$ in $X$ is given. Denote the embedding by $i: Q \hookrightarrow X$. To simplify the notations, we will still use the same notations $\LLieg$, $\LLieg_t$, $\LLiek$, $\LLieb$, $\LLien$, etc., for their restrictions to $Q$ as $\OO_Q$-sheaves. Also for a sheaf of right $\OO_X$-modules $\mathcal{E}$ and any sheaf $\mathcal{F}$ of left $\OO_Q$-modules, we will write $\mathcal{E} \otimes \mathcal{F}$ for $i^{-1}(\mathcal{E}) \otimes_{i^{-1}\OO_{X}} \mathcal{F}$. We define $\DhQ$ to be the image subsheaf of the composition of maps over $Q$
 \[
    i^{-1} (\U \LLieg_t) \otimes \omega_{Q/X} \hookrightarrow  i^{-1} (\U \LLieg[t]) \otimes \omega_{Q/X} \to i^{-1} (\D_{\Lieh} [t]) \otimes \omega_{Q/X}.
 \]
Just like $\D^{\lambda}_{X \gets Q}$, the sheaf $\DhQ$ is a left $i^{-1}(\U \LLieg_t)$-module and a right $\U (\LLiek[t])$-module. The difference is that we did not quotient out $\Dh$ by any ideal before pulling it back to $Q$. We are going to do it now. The reason for the difference is that, as we will see below, the construction of such quotient depends on $Q$ (or more precisely, the associated $K$-conjugacy classes of Cartan subalgebras) so it cannot be carried out uniformly over the entire flag variety $X$.

The choice of the $K$-orbit $Q$ determines an involution $\theta_Q$ on the abstract Cartan subalgebra $\Lieh$ (\cite{Milicic}), which gives rise to a decomposition $\Lieh = \Liet_Q \oplus \Liea_Q$, where $\Liea_Q$ is the $(-1)$-eigenspace of $\theta_Q$. Now consider over $Q$ the sheaf
   \[  \LLieh_{Q,t} = \LLieg_t \cap \LLieh[t] = \LLiet_Q[t] \oplus t \LLiea_Q[t].  \]
This is a constant vector bundle over $Q$ and so its fiber $\Lieh_{Q,t}$ can be regarded as \emph{the contraction of the universal Cartan subalgebra $\Lieh$}. Moreover, it carries a trivial Lie algebroid structure, even though $\LLieg|_Q$ does not. The sheaf of commutative algebras $\U \LLieh_{Q,t} = S \LLieh_{Q,t}$ acts on $\DhQ$ from the right.

\begin{comment}
Set $\LLieb_t:=\LLieg_t \cap \LLieb[t]$ and $\LLien_t:=\LLieg_t \cap \LLien[t]$. Since over the $K$-orbit the sheaves $\LLieb \cap \LLiek$ and $\LLien \cap \LLiek$ have geometric fibers with constant ranks and hence are locally free, $\LLieb_t$ and $\LLien_t$ are also locally free over $Q \times \complex$. The fiber $\Lieb_{0,x}$ of the restriction $\LLieb_0 = \LLieb_t |_{Q \times \{ 0 \} }$ can be thought of as the contraction of the Borel subalgebra $\Lieb_x$ to a subalgebra of $\Lieg_0$. Analogously for $\LLien_0$. We define $\LLieh_{Q,t} = \LLieb_t / \LLien_t$. It is canonically isomorphic to the trivial vector bundle $\LLieh|_Q [t]$ over $Q \times \complex$. The choice of the $K$-orbit $Q$ determines an involution $\theta_Q$ on the abstract Cartan subalgebra $\Lieh$, which gives rise to a decomposition $\Lieh = \Liet_Q \oplus \Liea_Q$. We have a natural isomorphism $m_t: \LLieh[t] = \LLiet_Q[t] \oplus \LLiea_Q[t] \to \LLieh_{Q,t} = \LLiet_Q[t] \oplus t \LLiea_Q[t]$ by multiplying the $\LLiea_Q[t]$-part by $t$.
\end{comment}

Any given $\lambda \in \Lieh^*$ determines a character 
  \[  \lambda_t : = \lambda_c + (\lambda_{nc} / t) - \rho \]
of $\Lieh_{Q,t}$ (as an element in $\Liet_Q^*[t] \oplus t^{-1}  \Liea^*_Q [t]$), where $\lambda_c = \lambda|_{\Liet_Q}$, $\lambda_{nc} = \lambda|_{\Liea_Q}$.
 We then set 
  \[ \DlQ:=\DhQ \otimes_{\U \LLieh_{Q,t}} (\complex[t])_{\lambda_t}. \]
$\DlQ$ inherits a right $\U \LLiek$-module structure from that of $\DhQ$. We have
an isomorphism
  \begin{equation}
     \DlQ \simeq \D^{\lambda_t}_{X \gets Q} = (i^{-1} \D_{\lambda_t})\otimes \omega_{Q/X}, ~\forall~ t \neq 0,
   \end{equation}
over $Q$, where the right hand side is exactly the transfer bimodule \eqref{eq:transfer} used to define the direct image functor of twisted $\D$-modules.  

Now for any $K$-homogeneous connection $\phi$ on $Q$ which is compatible with $(\lambda - \rho)|_{\Liet_Q}$, we set 
\begin{equation}\label{eq:pushfwd_Dt}
  \I(\lambda_t,Q,\phi) := i_* \left(\D^{\lambda_t}_Q \otimes_{\U \LLiek} \phi \right) = i_* \D^{\lambda_t}_Q \otimes_{\U \LLiek} i_* \phi. 
\end{equation}
It is a left $\U \LLieg_t$-module. The specialization of $\I(\lambda_t,Q,\phi)$ to $t = 1$ is a $\D_\lambda$-module over $X$, therefore its global sections form a $(\Lieg, K)$-module with infinitesimal character $\chi_\lambda$. When $\lambda$ is regular and dominant, it has a unique irreducible submodule, denoted by $M(\lambda,Q,\phi)$. The global sections of the specialization $\I_0(\lambda,Q,\phi) := \I(\lambda_t,Q,\phi) |_{t = 0}$ give a $(\Lieg_0,K)$-module, denoted by $\widetilde{M}_0(\lambda,Q,\phi)$. Since $S(\Lies)$ is a subalgebra of $\U\Lieg_0$, the module $\widetilde{M}_0(\lambda,Q,\phi)$ can be regarded as a $K$-equivariant coherent sheaf $\E=\E(\lambda,Q,\phi)$ over $\Lies^*$. For generic $\lambda$, the support $\supp(\E)$ of $\E$ in $\Lies^*$ is a single $K$-orbit, over which $\E$ is a vector bundle, and $\widetilde{M}_0(\lambda,Q,\phi)$ is irreducible as a $(\Lieg_0,K)$-module. 

\begin{comment}
In general, $\supp(\mathcal{L})$ might consist of more than one $K$-orbit. We expect it contains a unique minimal closed $K$-orbit $O_{min}$. The restriction of $\widetilde{M}_0(\lambda,Q,\phi)$ to $O_{min}$ is then an irreducible  $(\Lieg_0,K)$-module, denoted as $M_0(\lambda,Q,\phi)$.
\end{comment}

However, special care needs to be taken when $\lambda_{nc}=\lambda|_{\Liea_Q}$ is not regular. One reason is that Beilinson-Bernstein localization theorem does not behave well for singular values of $\lambda$. Another main reason is that, as we will see in the example of $SL(2,\real)$, the sheaf $\E$ over $\Lies^*$ might not be coherent and $\widetilde{M}_0(\lambda,Q,\phi)$ might not be finitely generated as a $U \Lieg_0$-module. So we need to adjust the definition \label{eq:pushfwd_Dt} of $\I(\lambda_t,Q,\phi)$. For instance, when $\lambda=0$ we expect the adjusted construction gives a coherent sheaf supported on $\mathcal{N} \cap \Lies^*$, where $\mathcal{N}$ is the nilpotent cone in $\Lieg^*$. This indicates we need to apply the usual Rees module construction in this situation so that $\supp(\E)$ is the associated variety of the representation.  We expect that a mixture of our previous definition \eqref{eq:pushfwd_Dt} of $\widetilde{M}_0(\lambda,Q,\phi)$ and the Rees module construction will lead to a general definition of $\I(\lambda_t,Q,\phi)$ and hence of $\widetilde{M}_0(\lambda,Q,\phi)$. The new version of $\I(\lambda_t,Q,\phi)$ would be a subsheaf of the one in \eqref{eq:pushfwd_Dt}. We will check this in the case of $\SLR$.

Last but not least, when the degenerate situation happens, the module $\widetilde{M}_0(\lambda,Q,\phi)$ is expected to be reducible, since $\supp(\E)$ might consist of more than one $K$-orbit. We expect it contains a unique minimal closed $K$-orbit $O_{min}$. The restriction of $\widetilde{M}_0(\lambda,Q,\phi)$ to $O_{min}$ is then an irreducible  $(\Lieg_0,K)$-module, denoted as $M_0(\lambda,Q,\phi)$. Hence $M_0(\lambda,Q,\phi)$ is the minimal quotient of $\widetilde{M}_0(\lambda,Q,\phi)$. With the classification of tempered representations in terms of Harish-Chandra sheaves in mind (\cite{Chang}, \cite{Mirkovic}), we now state our conjecture.

\begin{conjecture}\label{conj:Mackey}
  For any irreducible tempered $(\Lieg,K)$-module realized as (global sections of) a standard Harish-Chandra module $\I(Q, \phi)$ associated to the triple $(\lambda, Q, \phi)$, we can define a sheaf $\I(\lambda_t,Q,\phi)$ of coherent $(\Lieg_t,K)$-modules over the flag variety, such that its specialization to $t=1$ is $\I(Q,\phi)$ and (the global sections of) its specialization to $t=0$ is a coherent $(\Lieg_0,K)$-module $\widetilde{M}_0(\lambda,Q,\phi)$. Each $\widetilde{M}_0(\lambda,Q,\phi)$ has a unique irreducible quotient $M_0(\lambda,Q,\phi)$ such that the correspondence $M(\lambda,Q,\phi) \longleftrightarrow M_0(\lambda,Q,\phi)$ realizes Afgoustidis' Mackey-Higson bijection between the tempered duals of $G_\real$ and $G_{0,\real}$. Moreover, this bijection can be extended to all admissible duals of $G_\real$ and $G_{0,\real}$. 
\end{conjecture}

The possibly reducible $(\Lieg_0,K)$-module $\widetilde{M}_0(\lambda,Q,\phi)$ remembers more information about the original $(\Lieg,K)$ than the irreducible $(\Lieg_0,K)$-module $M(\lambda,Q,\phi)$. The process of going from $\widetilde{M}_0(\lambda,Q,\phi)$ to the corresponding $(\Lieg,K)$ can be interpreted as deformation quantization of $\widetilde{M}_0(\lambda,Q,\phi)$ as a $K$-equivariant sheaf over its support in $\Lies^*$ as a Lagrangian subvariety in certain coadjoint orbit of $G$. This viewpoint will be explored elsewhere.

\section{The example of $SL(2, \real)$}\label{sec:sl2}

We now examine Conjecture \ref{conj:Mackey} in the example of $\SLR$. Since $\SLR$ is conjugate to $\SU$ in its complexification $\SLC$ and the formulas become simpler for $\SU$, we will consider the Lie groups
\begin{equation}\label{eq:groups}
\begin{gathered}
  G_\real = \SU = \left\{  
    \begin{pmatrix}
      a  &  b \\
      \bar{b}  & \bar{a}
    \end{pmatrix}
    \bigg| ~ a, b \in \complex, |a|^2 - |b|^2 = 1
   \right\}, \\
  G = \SLC = \left\{
    \begin{pmatrix}
      a & b \\
      c & d
    \end{pmatrix}
    \bigg| ~ a, b, c, d \in \complex, ad - bc = 1
  \right\}, \\   
  K_\real = \left\{
    \begin{pmatrix}
      a & 0 \\
      0 & a^{-1}
    \end{pmatrix} 
    \bigg| ~ a \in \complex, |a|=1 
  \right\}, 
  \quad
  K = \left\{
    \begin{pmatrix}
      a & 0 \\
      0 & a^{-1}
    \end{pmatrix}
    \bigg| ~ a \in \complex^*
  \right\}.
\end{gathered}
\end{equation}
The complexified Lie algebra of $G_\real$ is
\begin{equation}  
  \Lieg = \slc = \left\{
    \begin{pmatrix}
      p & r \\
      s & -p
    \end{pmatrix}
    \bigg| ~ p, r, s \in \complex
  \right\}.
\end{equation}
The Cartan involution $\theta$ on $\Lieg$ is given by $\theta(T) = J T J$, $T \in \Lieg$, where
\[ 
  J=
    \begin{pmatrix}
      -1 & 0 \\
      0 & 1
    \end{pmatrix}.
  \]
The Lie algebra $\Lieg$ has a standard basis
\begin{equation}
  E = \begin{pmatrix}
      0 & 1 \\
      0 & 0
    \end{pmatrix}, \quad
  F = \begin{pmatrix}
      0 & 0 \\
      1 & 0
    \end{pmatrix},   \quad
  H = \begin{pmatrix}
      1 & 0 \\
      0 & -1
    \end{pmatrix}
\end{equation}
satisfying the standard commutation relations
\begin{equation}\label{eq:comm}
  [ H, E ] = 2 E, \quad [H, F] = -2 F, \quad [E, F] = H.
\end{equation}
Moreover, $H$ spans the Lie algebra $\Liek$ of $K$ and $\{ E, F \}$ spans $\Lies$. We set
  \[  E_t = tE, \quad F_t = tF, \quad H_t = H \in \Lieg_t. \]

The flag variety $X$ of $G=\SLC$ is the variety of full flags of $\complex^2$, so $X$ can be identified with $\PP$ with homogeneous coordinates $[z_0 : z_1]$. We set $\infty=[1:0]$, so we can choose coordinate $z$ on $U_0 = \PP-\{ \infty \}$ by $z([z_0:1]) = z_0$ so that $U_0$ is identified with $\complex$. We also set $0=[0:1]$ and $U_1 = \PP - \{0\}$ and choose the coordinate $\zeta$ on it by setting $\zeta([1:z_1])=z_1$. The $K$-action decomposes $X$ into three orbits: $0$, $\infty$ and $Q=\PP-\{ 0, \infty \} \simeq \complex^*$. The root system is $\Sigma=\{ \alpha, -\alpha \}$ where $\alpha$ denotes the positive root. Set $\rho = \frac{1}{2} \alpha$ and $l = \alpha\check{}~(\lambda)$, where $\alpha\check{} \in \Lieh$ is the dual root of $\alpha$, so that $\alpha\check{}~(\rho)=1$ and $\lambda = l \rho$. At $\infty$, the corresponding Borel subgroup and Borel subalgebra are
\[  B_\infty 
    = \left\{ \begin{pmatrix}
      a & b \\
      0 & d
    \end{pmatrix} \right\},
    \quad
    \Lieb_\infty
    = \left\{ \begin{pmatrix}
      p & r \\
      0 & -p
    \end{pmatrix} \right\}
    = \complex H \oplus \complex E,
 \]
respectively. Similarly, the Borel subgroup and Borel subalgebra associated to $0$ are
\[  B_0 
    = \left\{ \begin{pmatrix}
      a & 0 \\
      c & d
    \end{pmatrix} \right\},
    \quad
    \Lieb_0
    = \left\{ \begin{pmatrix}
      p & 0 \\
      s & -p
    \end{pmatrix} \right\}
    = \complex H \oplus \complex F,
 \]
 respectively. The Lie subalgebra $\Liek=\complex H$ forms a $\theta$-stable Cartan subalgebra of $\Lieg$.  If we specialize $\Liek$ at $0$, the vector $H$ corresponds to the dual root $\alpha\check{}$, while $H$ corresponds to $-\alpha\check{}$ when specialized at $\infty$, .

The action of $G$ on $X$ is given by
\begin{equation}
  g = \begin{pmatrix}
      a & b \\
      c & d
    \end{pmatrix}
    : \quad [ z_0 : z_1 ]  \mapsto [a z_0 + b z_1 : c z_0 + d z_1],
\end{equation}
hence the induced infinitesimal acton $\tau: \LLieg \to TX$ under the $z$-coordinate and $\zeta$-coordinate is given by
\begin{equation}\label{eq:action}
E \mapsto - \dz = \zeta^2 \dzz, \quad  F \mapsto z^2 \dz = - \dzz, \quad H \mapsto -2 z \dz = 2 \zeta \dzz. 
\end{equation}
For example, $E$ acts on any function $f(z)$ by taking derivative of $f(\exp(-tE)z) = f(z-t)$ with respect to $t$, which gives $(Ef)(z) = -z \dz f (z)$. $H$ correspsonds to $-2z \dz$ since if we identify $K$ with $\complex^*$ by taking the first entry in the diagonal matrix expression in \eqref{eq:groups}, any $a \in K \simeq \complex^*$ acts on the chart $(U_0, z)$ by multiplication by $a^2$.

Over $U_0$ and $U_1$ we have two trivializations of $\D_\lambda$ such that, under $z$-coordinate and $\zeta$-coordinate, the map $\LLieg \to \D_\lambda$ is given by
\[ E \mapsto - \dz = \zeta^2 \dzz - (l-1) \zeta, \quad   F \mapsto z^2 \dz - (l-1) z = -\dzz, \quad H \mapsto - 2 z \dz + (l-1) = 2 \zeta \dzz - (l-1). \]

\subsection{Discrete series and limits of discrete series representations}\label{subsec:discrete} 
The global sections of the standard Harish-Chandra sheaves at the two closed $K$-orbits $\{0\}$ and $\{\infty\}$ represent the discrete series representations and limits of discrete series. Take the orbit $\{0\}$ for example. In this case $\Liek$ is the unique maximally toroidal $\theta$-stable Cartan subalgebra and the full $K$-group is the stabilizer of the $K$-action. Moreover, the nilpotent ideal $\Lien^-_0$ of the corresponding Borel subalgebra $\Lieb_0$ is spanned by $F$, while $\Lien^+_0$ is spanned by $E$, both contained in $\Lies$. Hence for $\lambda= n \rho$ with $n \in \ZZ$, $n \geq 0$, we have $\lambda_t = \lambda - \rho$, which does not depend on $t$. The choice of the integrable connection over $\{0\}$ is unique, which is $\complex_{\lambda - \rho}$. Since $\omega_X = \OO(-2\rho) = \OO_{\PP}(-2)$, the sheaf $\I(\lambda_t,Q,\phi)$ in \eqref{eq:pushfwd_Dt} is given by the vector space
 \[ \D^+_{n,t} = \D^{\lambda_t}_{ \{0\} } = \U ( t \Lien^{+}_0 [t] ) \otimes_\complex \omega^{-1}_{X}|_{\{0\}} \otimes_\complex \complex_{\lambda - \rho} \simeq \U ( t \Lien^{+}_0 [t] ) \otimes_\complex  \complex_{\lambda - \rho + 2 \rho} \simeq \complex[E_t] \delta_{n} \otimes_\complex \complex[t]  \]
supported at $0$, where $\delta_n=1 \in \complex_{\lambda+\rho}$. The $\Lieg_t$-module structure on $D^+_{n,t}$ satisfies the relations
\[ H_t \delta_n = H \delta_n = [\alpha\check{}~(\lambda+\rho)]\delta_n = (n+1)\delta_n, \quad F_t \delta_n = 0, \quad [E_t, F_t] = 2 t^2 H_t. \]
Therefore the $K$-weight of $\delta_n$ is $n+1$ and $\D^+_n = \D^+_{n,t}/(t-1)\D^+_{n,t} = \complex[E] \delta_n$ is the Harish-Chandra module of the holomorphic disrete series representation $D^+_n$ with Harish-Chandra parameter $n$ and minimal $K$-type of weight $n+1$ (or the limit of discrete series $D^+_0$ when $n=0$). In this case, $\D^+_{n,t}$ concides with the Rees module of the $\D$-module $\D^+_n$ with respect to its natural $K$-invariant filtration
  \[ F_p \D^+_{n} = \spn_\complex \{ E^k \delta_n | 0 \leq k \leq p  \}, ~ \forall~ p \geq 0, \quad \text{and} \quad F_p \D^+_{n} = 0, ~\forall~ p < 0. \]
Specialize $\D^+_{n,t}$ to $t=0$, we get the $(\U\Lieg_0, K)$-module $\D^+_{n,0} = \complex[E] \delta_n$ such that its $\U \Lieg_0$-module structure is given by
  \[ H \delta_n = (n+1)\delta_n, \quad F \delta_n = 0.  \]
The support of $\D^+_{n,t}$ as an $S \Lies$-module is the subvariety $O^+$ in $\Lies^*$ is determined by the equation $F=0$, which lies in the nilpotent cone $\mathcal{N}$ in $\Lieg^*$ and is isomorphic to the complex line $\complex$. The module $\D^+_{n,0}$ consists of the global sections of the $K$-equivariant line bundle over $O^+$, of which $K=\complex^*$ acts on the fiber at the origin by weight $(n+1)$. $\D^+_{n,0}$ is not an irreducible $(\U \Lieg_0,K)$-module, since $O^+$ consists of two $K$-orbits: the origin as the unique closed orbit and its complement. Therefore $\D^+_{n,0}$ has a unique irreducible quotient corresponding to the restriction of the line bundle to the origin, which gives the representation $\complex_{n+1}$, which is $\complex$ with the weight $n+1$ action of $K$ and the trivial $S \Lies$-module structure. Hence the Mackey bijection relates (the Harish-Chandra module of) the holomorphic discrete series representation of $G_\real$ with minimal $K$-type of weight $n+1$ with its minimal $K$-type $\complex_{n+1}$ as a one dimensional representation of the motion group $G_{\real,0}$.

Apply a similar construction to the orbit $\infty$ and $\lambda= n \rho$ ($n \in \ZZ$, $n \geq 0$), we obtain the deformation of anti-holomorphic discrete series representation $D^-_{n}$ for $n \geq 1$ (and the limit of discrete series representation $D^-_{0}$) to a line bundle supported on the $K$-subvariety $O^- \subset \Lies^*$ cut out by the equation $E=0$. The restriction of the line bundle to the origin is the one dimensional motion group representation $\complex_{-n-1}$ under the Higson-Mackey bijection.

\subsection{Principal series representations}\label{subsec:principal}
Now we study the principal series representations associated to the unique open $K$-orbit $Q$. The stabilizer in $K$ of any point in $Q$ is $\{ \pm 1 \}$. It is useful to work with a second trivialization of $\D_\lambda$ restricted on the open $K$-orbit $Q \simeq \complex^*$ by restricting the original $z$-trivialization on $U_0$ to $Q \simeq \complex^*$ and twisting it by the automorphism of $\D_{\complex^*}$ induced by
  \begin{equation}\label{eq:trivialization}
    \dz \mapsto \dz + \frac{l-1}{2z} = z^{-\frac{l-1}{2}} \dz z^{\frac{l-1}{2}}. 
  \end{equation}
Now the map $\LLieg \to \D_\lambda |_{Q} \simeq \D_{\complex^*}$ is given by
  \begin{equation}\label{eq:second_trival}
    E \mapsto -\dz - \frac{l-1}{2z}, \quad F \mapsto z^2 \dz - \frac{l-1}{2} z, \quad H \mapsto  -2z \dz, 
  \end{equation}
such that $H$ still corresponds to the original differential operator generated by the $K$-action on $Q \simeq \complex^*$ as in \eqref{eq:action}. In other words, The trivialization is via the canonical isomorphism $\U \LLiek |_Q \simeq \D_\lambda|_Q$ induced by the composition $\U \LLiek \hookrightarrow \D_\Lieh \twoheadrightarrow \D_\lambda$ (since the stabilizer of the $K$-action on $Q$ is discrete in this case), so that the expression for $H$ does not change when $\lambda$ varies.

Now we replace $\lambda$ by $\lambda/t$, or equivalently, $l$ by $l/t$, to get the family of sheaves $\D_{\lambda/t}$ so that $(E,F,H)$ satisfy \eqref{eq:second_trival} with $l$ replaced by $l/t$ for each $t \in \complex^*$. We then rescale the resulting differential operators for $E$ and $F$ by multiplying by $t$ and get 
  \begin{equation}\label{eq:efh_t}
     E_t = t \left(-\dz + \frac{1}{2z} \right) - \frac{l}{2z}, \quad F_t = t \left( z^2 \dz + \frac{1}{2}z \right) - \frac{l}{2} z, \quad H_t = H = -2z \dz, 
   \end{equation}
for any $t \in \complex$. $(E_t, F_t, H_t)$ still satisfies the commutation relations
\begin{equation}
  [ H_t, E_t ] = 2 E_t, \quad [H_t, F_t] = -2 F_t, \quad [E_t, F_t] = t^2 H_t. 
\end{equation}
Moreover, they satisfy the equation
\begin{equation}\label{eq:casimir_t}
  \frac{t^2}{4}H^2_t + \frac{1}{2}(E_t F_t + F_t E_t) = \frac{l^2-t^2}{4},
\end{equation}
of which the left hand side becomes the standard Casimir operator of $\slc$ when $t=1$. Hence when $t=0$, 
\begin{equation}\label{eq:efh0}
  E_0 = -\frac{l}{2z}, \quad F_0 =  -\frac{l}{2} z, \quad H_0 = -2z \dz,
\end{equation}
and $(E_0, F_0, H_0)$ forms a basis of the Lie algebra $\Lieg_0$ of the motion group. The kernel of the homomorphism $S(\Lies) \to \Gamma(\complex^*, \D|_{\complex^*}) \simeq \Gamma(Q, \D^{\lambda_t}_Q)$ ($\omega_{Q/X}$ is trivial in this case) is generated by the relation 
  \[  E_0 F_0 = \frac{1}{2}(E_0 F_0 + F_0 E_0)  =  \frac{-l}{2z} \cdot \frac{-l}{2} z = \frac{l^2}{4}, \] 
which can also be deduced from \eqref{eq:casimir_t} by setting $t=0$. 

Now we assume $\lambda$ is regular (and not necessarily dominant), i.e., $l \neq 0$. Then the $K$-orbit $C_\lambda$ in $\Lies^*$ corresponding to principal series representations with infinitesimal character $\lambda$ is the intersection of $\Lies^*$ with the regular semisimple coadjoint orbit in $\Lieg^*$ determined by the equation
  \begin{equation}\label{eq:orbit}
     \frac{1}{4} H^2 + EF = \frac{l^2}{4}
  \end{equation} 
 in $S \Lieg$, where $E, F, H \in \Lieg$ are considered as linear functions on $\Lieg^*$ (since $\Lies^* \subset \Lieg^*$ is given by the equation $H=0$). Note that the calculation above does not depend on which $K$-homogenous connection we choose on $Q$. In other words, for a fixed $\lambda$ we alway get the same $K$-orbit in $\Lies^*$. In fact, the image of the homomorphism $S \Lies \to \Gamma(Q, \D^{\lambda_t}_Q|_{t=0})$ lies in $\Gamma(Q, \OO_Q)$ by \eqref{eq:efh0}, which gives an isomorphism 
   \[ \Phi_\lambda: Q \to C_\lambda,\] 
of affine varieties.
 
Now we consider $K$-homogeneous connections over $Q$. Since the isotropy subgroup in $K$ for any point of $Q$ is $\{ \pm 1 \}$. Denote by $\phi_0$ the irreducible $K$-homogeneous connection corresponding to the representation $x \mapsto x^k$ of the group $\{ \pm 1 \}$, $k=0,1$. Under the new trivialization \eqref{eq:trivialization}, the space of global sections of $\phi_k$ on $Q \simeq \complex^*$ is spanned by the formal sections $z^{n+\frac{k}{2}}$, $n \in \ZZ$. Since $Q$ is open, the natural map $\U\LLiek[t] \to \D^{\lambda_t}_Q$ is surjective. Let us first follow \eqref{eq:pushfwd_Dt} and pushforward $\phi_k$ to $X$ to get the $\U\LLieg_t$-module 
  \[  \I(\lambda_t, Q, \phi_k) = i_* \phi_k [t].\] 
In this case the family of twisted $\D$-modules is constant as $\OO_X$-sheaves. The only thing that is changing is the action of the differential operators. Specialize to $t=0$, we get the $\U\LLieg_0$-module $\I_0(\lambda, Q, \phi)=i_*\phi_k$, on which $\Lieg_0$ acts by multiplying by functions in \eqref{eq:efh0}. Its space of global sections $\widetilde{M}_0(\lambda, Q, \phi)=M_0(\lambda,Q,\phi)$ gives an irreducible $(\Lieg_0,K)$-module. We have already worked out the corresponding $K$-orbit $C_\lambda$ in $\Lies^*$ in \eqref{eq:orbit}. The line bundle over $C_\lambda$ is just the pullback of $\phi_k$ via the isomorphism $C_\lambda \simeq Q$. 

Specialization of $\I(\lambda_t, Q, \phi_k)$ to $t=1$ gives the standard Harish-Chandra sheaf $\Prin_{\lambda,k} := \I(\lambda, Q, \phi_k)$ which corresponds to infinite dimensional principal series representations. We get the even principal series representations and the odd principal series representations when $k=0$ and $k=1$ respectively. However, $\I(\lambda, Q, \phi_k)$ might be reducible certain values of $\lambda$ and $k$. It is reducible if and only if $l+k = \alpha\check~(\lambda) + k$ is an odd integer and in this case it contains $\OO(\lambda-\rho)$ as the unique irreducible $\D_\lambda$-module over $X$ (\cite{Milicic}), where $\OO(\lambda-\rho) = \OO_{\PP}(l-1)$ is the sheaf of meromorphic functions over $X$ with pole of order less or equal than $(l-1)$ at $0$ ($l-1 \geq 0$ since $\lambda$ is regular, integral and dominant). By the Borel-Weil-Bott theorem, the global sections of $\OO(\lambda-\rho)$ form the (nonunitary) irreducible finite dimensional representation of $G_\real$ of highest weight $\lambda-\rho$, which is $M(\lambda,Q,\phi_k)$ by our notation. So in this case, the infinite dimensional irreducible $(\Lieg_0,K)$-module $M_0(\lambda,Q,\phi_k)$ corresponds to the finite dimensional $M(\lambda,Q,\phi_k)$ under the Mackey-Higson bijection.

\begin{comment}
Let us consider the enhanced definition \ref{defn:newI} of $\I(\lambda, Q,\phi_k)$. We shall see how the choice of $\M$ would affect the family $\I(\lambda, Q, \phi_k)$. Note that when $l$ is not real, the sheaf $i_* \phi_k$ as a $\D_{\lambda/t}$-module is always irreducible and hence is generated by any $\M$ for $t \neq 0$. Thus $\M[t]$ generates the entire $\I(\lambda, Q, \phi_k)=i_* \phi_k[t]$.
\end{comment}

\subsection{Principal series representations with singular character}
Now consider the case when $\lambda$ is singular, i.e., $\lambda=0$. If we still follow the original definition \eqref{eq:pushfwd_Dt}, the operators $E_0$ and $F_0$ will act as the zero operator on $\I_0(\lambda, Q, \phi)$ by \eqref{eq:efh0} and the correpsonding sheaf $\mathcal{L}$ on $\Lies^*$ will be supported at the origin. But the representation space is infinite dimensional so the sheaf is not coherent. The $(\Lieg_0,K)$-module is highly reducible and each of its $K$-component is an irreducible quotient, so there is no canonical way to pick out one. Hence we need to treat this case differently.

In the case of the even principal series representation with $\lambda=0$, $\phi_0$ is the trivial line bundle over $K$ and $\Prin_{0,0}=i_* \phi_0$ is irreducible. We take the sheaf of $\OO_X$-modules on $X$ spanned by the minimal $K$-type of $\Prin_{0,0}$, i.e.,
  \[ \M = \spn_{\OO_X} \{z^0 = 1 \} = \OO_X \subset i_* \phi_0. \]
Then the sheaf of $\U \LLieg_t$-submodules in $i_* \phi_0[t]$ generated by $\M$ coincides with the usual Rees-module of the $\D_0$-module $\Prin_{0,0}$. We still denote it by $ \I(\lambda_t,Q,\phi_0)$. Namely, choose the $K$-invariant filtration of $\Prin_{0,0}$ to be
\begin{equation}\label{eq:filtr_prin}
     F_p \Prin_{0,0}  = \spn_{\OO_X} \left\{ z^{n} | -p \leq n \leq p  \right\}, ~ \forall~ p \geq 0, \quad \text{and} \quad F_p \Prin_{0,0} = 0, ~\forall~ p < 0,
  \end{equation}
so that $F_0 \Prin_{0,0} = \M$. Then
\begin{equation}\label{eq:prin_sing_even}
   \I(\lambda_t,Q,\phi_0) = \sum_{p \geq 0} t^p F_p \Prin_{0,0}[t] \subset i_* \phi_0 [t].
 \end{equation}
The restriction to $t=0$ is then the associated graded sheaf of $\Prin_{0,0}$ with respect to the filtration $F_p \Prin_{0,0}$,
\begin{equation}\label{eq:prin_sing}
   \I_0(0,Q,\phi_0) \simeq \bigoplus_{p \geq 0} t^p (F_p \Prin_{0,0} / F_{p-1} \Prin_{0,0}).
 \end{equation}
Since all $F_p \Prin_{0,0}$ are spanned by subsets of the natural basis $\{ z^n| n \in \ZZ \}$ of $M(\lambda,Q, \phi_0)=\Gamma(X,\Prin_{0,0})$ which are $K$-components, we have the $K$-equivariant identification 
  \[  \widetilde{M}_0(\lambda,Q,\phi_0) = \Gamma(X,\I_0(0,Q,\phi_0)) \simeq \spn_\complex \{ t^{|n|}z^n ~|~ n \in \ZZ \}, \]
such that the actions of $E_0$ and $F_0$ are given by
\begin{equation}
\begin{aligned}
  E_0 (t^{n} z^{-n}) = \left( n+\frac{1}{2} \right) t^{n+1} z^{-n-1}, && \forall~ n \geq 0, \\ E_0(t^n z^n) = 0, && \forall ~ n > 0,  \\
   F_0 (t^{n} z^{n}) = \left( n+\frac{1}{2} \right) t^{n+1}  z^{n+1}, && \forall~ n \geq 0, \\ F_0(t^{n} z^{-n}) = 0, && \forall ~ n > 0.
\end{aligned}
\end{equation}
One can see this by plugging $l=0$ into \eqref{eq:efh_t}. The corresponding coherent sheaf $\LL$ over $\Lies^*$ is locally of free of rank $1$. Its support is the intersection of the nilpotent cone $\N$ in $\Lieg^*$ with $\Lies^*$, which is the union of the two complex lines determined by the equations $E=0$ and $F=0$ in $\Lies^*$, respectively. The unique quotient of $\widetilde{M}_0(\lambda,Q,\phi_0)$ is the same as the geometric fiber of $\LL$ at $0 \in \Lies^*$, so we have
  \[  M_0(\lambda,Q, \phi_0) = \complex_0, \]
i.e., a copy of $\complex$ with the trivial $K$-action and the trivial $S(\Lies)$-action. This coincides with the prediction of Mackey-Higson bijection, or more specifically, Vogan's bijection, which relates the even principal series representation of $G_0$ with $\lambda=0$ to one-dimensional $G_0$-representation $\complex_0$.

The odd principal series representation $\Prin_{0,1}$ is reducible. Its global sections decompose into a direct sum of two limit of discrete series representations, which have already been discussed above. There are two minimal $K$-types, $z^{\pm \frac{1}{2}}$, each of which belongs to one of the two limit of discrete series representations. Similar to the even principal series, the subsheaf of $\U \LLieg_t$-submodules $ \I(\lambda_t,Q,\phi_1)$ in $i_* \phi_1[t]$ generated by $\{ z^{\pm \frac{1}{2}}\}$ coincides with the usual Rees module of the $\D_0$-module $\Prin_{0,1}$ of which the filtration is generated by $\{ z^{\pm \frac{1}{2}}\}$. Furthermore, this Rees-module decomposes as $(\Lieg_t, K)$-module into a direct sum of the Rees modules  $\D^+_n$ of the two limit of discrete series representations discussed in \ref{subsec:discrete}. Hence the decomposition also holds for the $(\Lieg_0,K)$-module $\I_0(\lambda,Q,\phi_1)$ when specialized at $t=0$. This a family version of the usual Schimd identities (see, e.g., \cite{Localization}).

\subsection{A new deformation generated by minimal K-types}\label{subsec:new_family}

In fact, the special treatment of principal series representations with $\lambda=0$ and our original definition of the deformation can be unified to a new definition of $\I(\lambda_t,Q,\phi)$. Let $\M$ be the sheaf of $\OO_X$-submodules in $i_* \phi_k$ generated by the minimal $K$-types of $\Gamma(X, \Prin_{0,k})$. That is,
  \[ \M = \M(Q,\phi) = \left\{ \begin{array}{ll}
                                     \spn_{\OO_X} \{z^0 = 1 \} = \OO_X \subset i_* \phi_0, & k=0, \\
                                     \\
                                     \spn_{\OO_X} \{z^{\pm \frac{1}{2}} \} = \OO(\rho) \subset i_* \phi_1, & k=1.
                                    \end{array}
                            \right.\] 
Define $\I(\lambda_t, Q ,\phi_k)$ to be the sheaf of $\U\LLieg_t$-submodule in the original $\I(\lambda_t,Q,\phi_k)$ in \eqref{eq:pushfwd_Dt} generated by $\M$. One can show that
\begin{equation}\label{eq:newI}
  \I(\lambda_t, Q ,\phi_k) = (i_* \DlQ) \otimes_{\U\LLiek} \M,
\end{equation}        
where the tensor product is over $X$. We still write $\widetilde{M}_0(\lambda,Q,\phi)$ as the global sections of $ \I(\lambda_t, Q ,\phi_k)$ and  

We compare the new defined $\I(\lambda_t, Q ,\phi_k)$ with the original one. 

\begin{comment}
\begin{definition}
  Let $\complex[x]$ be the algebra of polynomials in real variable $x$ with complex coefficients. We have the obvious map $\complex[t] \to \complex[x]$ sending $t \in \complex$ to its real part $x$ Define the restriction of $\I(\lambda_t, Q ,\phi_k)$ to an algebraic family over $\real \subset \complex$ to be
   \begin{equation}\label{eq:Ft}
    \I_\real(\lambda_t, Q ,\phi_k) =\I(\lambda_t, Q ,\phi_k) \otimes_{\complex[t]} \complex[x]).
  \end{equation}
\end{definition}     
\end{comment}

\begin{proposition}\label{prop:newI}
  In the case of principal series representations $\Prin_{\lambda,k}$ such that $\lambda=\l \rho$ and $l$ has nonzero imaginary part, we have  
  \[  \I(\lambda_t, Q ,\phi_k)|_{t=x} = i_* \phi_k, \]
  for any $x \in \real \subset \complex$. That is, the new version of $\I_\real(\lambda_t, Q ,\phi_k)$ coincides with the original one when restricted to a family over $\real \subset \complex$. 
\end{proposition}

\begin{proof}
  We need to show that $z^{n \pm \frac{k}{2}} \in \I_\real(\lambda_t, Q ,\phi_k)$ for any $n \in \ZZ$ and $k= 0, 1$. This can be done by induction on $|n|$. The case of $n=0$ is easily deduced from the definition of $\M$. Now assume it is true when $|n| \leq m$ for some $m \in \N$. Then by \eqref{eq:efh_t}, we have
  \begin{equation}\label{eq:Ft}
    F_t z^{m + \frac{k}{2} } = \left[ t \left( m + \frac{k+1}{2} \right) - \frac{l}{2} \right] z^{m+1+\frac{k}{2}}.  
  \end{equation}
Since $\Img(l) \neq 0$, the coefficient on the right hand side of the equation above is invertible when $t = x \in \real$ and hence $z^{m+1+ \frac{k}{2}} \in \I(\lambda_t, Q ,\phi_k)|_{t = x}$. Similarly we see that $z^{-m-1 - \frac{k}{2}} \in \I_\real(\lambda_t, Q ,\phi_k)|_{t=x}$ by computing $E_t z^{-m-\frac{k}{2}}$.
\end{proof}

The argument in the proof of Lemma \ref{prop:newI} fails when $t \in \complex$ since the coefficient in \eqref{eq:Ft} could be zero. If $l$ is a nonzero real number, this happens infinitely many times for $t \in \real$. 

\begin{proposition}\label{prop:prin_int}
  Suppose $l$ is a nonzero real number, then the specialization of smooth family $\I(\lambda_t, Q ,\phi_k)$ at each $t \in \real$ is the nonunitary principal series representation $\Prin_{l \rho /t, k}$ when $l / t \notin \N$ and is the irreducible finite dimensional representation of $\SLC$ of highest weight $(l/t-1)\rho$ when $l / t \in \N$.
\end{proposition}

\begin{proof}
Indeed, the specialization of $\I(\lambda_t, Q ,\phi_k)$ at $t$ is the $\OO_X$-submodule of $\I(l \rho /t , Q , \phi_k)$ generated by $\M$. When  $l/t \in \N$, the subsheaf is the line bundle $\OO(\lambda-\rho) = \OO_{\PP}(l/t-1)$. Otherwise the subsheaf is the entire $\I(l \rho /t , Q , \phi_k)$ since it is irreducible (\cite{Milicic}).
\end{proof}

We see that the new version of $\I(\lambda_t, Q ,\phi_k)$ differs from the original one when $l \in \real$. However, the specialization at $t=0$ and $t=1$ still gives the Mackey-Higson bijection. Indeed, when $l$ is a nonintegral real number, the specialization of the new family at $t=1$ is still the nonunitary principal series representation $\Prin_{l \rho ,k}$. When $l$ is a positive integer, the specialization of the new family at $t=1$ is already the finite dimensional subrepresentation $\Gamma(\PP, \OO_{\PP}(l-1))$ of $\Prin_{l \rho ,k}$. These cases coincide with the discussion at the end of \S ~ \ref{subsec:principal}. Moreover, we get the irreducible $(\Lieg,K)$-modules automatically.

Therefore we have extended Afgoustidis' Mackey-Higson bijection for tempered representations to admissible representations in the case of $\SLR$. 

\subsection{Intertwining operators}

There are intertwining functors among the categories of $\D_\lambda$-modules with different $\lambda$ in the Weyl group orbit (\cite{BB-Casselman}, \cite{Localization}). We show that they extend to isomorphisms among different $\I(\lambda, Q, \phi)$ and hence extend to isomorphisms among $(\Lieg_0,K)$-modules under the Mackey-Higson bijection.

Consider the principal series representation $\Prin_{\lambda,k} := \I(\lambda, Q, \phi_k)$, where $Q$ is the open orbit and $\lambda = l \rho$. Its global sections are spanned by the basis $e_p= z^{p+\frac{k}{2}}$, $p \in \ZZ$. Using the trivialization \eqref{eq:second_trival} of $\D_\lambda$ on $Q$ , we get
  \begin{equation}
    \begin{split}
      E e_p &= \left[ - \left(p+\frac{k}{2} \right) - \frac{l-1}{2}  \right] e_{p-1}  \\
      F e_p &= \left[ p + \frac{k}{2} - \frac{l-1}{2} \right] e_{p+1}  \\
      H e_p &= -2 \left(p+ \frac{k}{2} \right) e_p.
   \end{split}
  \end{equation}
Define rational functions $\alpha_p = \alpha_p(l)$, $p \in \ZZ$, such that $\alpha_{0} \equiv 1$ no matter when $k=0$ or $1$, and
  \begin{equation}\label{eq:alpha}
     \alpha_p =  \frac{- \left(p+\frac{k}{2} \right) + \frac{l+1}{2} }{- \left(p+\frac{k}{2} \right) + \frac{-l+1}{2} } \alpha_{p-1} = \frac{2p+k-l -1}{2p+k+l-1} \alpha_{p-1}. 
   \end{equation}
We form $f_p = \alpha_p e_p$, $p \in \ZZ$. In other words, we have
 \[ \alpha_p = \left\{ \begin{array}{ll}
                                    \displaystyle \prod_{j=1}^p \frac{2j+k-l-1}{2j+k+l-1}, & \quad p > 0, \\
                                     \\
                                    \displaystyle \prod_{j=1}^{|p|} \frac{2j-k-l-1}{2j-k+l-1}, & \quad p < 0.
                                    \end{array}
                            \right.\] 

This gives
   \begin{equation}\label{eq:intertwine}
    \begin{split}
      E f_p &= \alpha_p E e_p = \alpha_p \left[ - \left(p+\frac{k}{2} \right) - \frac{l-1}{2}  \right] e_{p-1} =  \left[ - \left(p+\frac{k}{2} \right) - \frac{-l-1}{2}  \right] f_{p-1}, \\
      F f_p &= \alpha_p F e_p = \alpha_p \left( p + \frac{k}{2} - \frac{l-1}{2} \right) e_{p+1}   = \left[ \left( p+ \frac{k}{2}  \right) - \frac{-l - 1 }{2} \right] f_{p+1}  \\
      H f_p &= -2 \left(p+ \frac{k}{2} \right) f_p.
   \end{split}
  \end{equation}
Since both $\Gamma(X,\Prin_{-\lambda,k})$ and $\Gamma(X,\Prin_{\lambda,k})$ are spanned by $e_p$, the equations \eqref{eq:intertwine} means that $A(\lambda,k) (e_p) = f_p$ defines a (normalized) intertwining operator 
  \[ A(\lambda,k): \Gamma(X,\Prin_{-\lambda,k}) \to \Gamma(X,\Prin_{\lambda,k})\]
of $(\Lieg,K)$-modules (\cite{Sally}). Now note that if we replace $l$ in the formulas above by $l/t$, the recursive formula \eqref{eq:alpha} for $\alpha_p$ will become
\begin{equation}\label{eq:alpha_t}
  \alpha_p(l/t)  = \frac{2p+k-l/t -1}{2p+k+l/t-1} \alpha_{p-1}(l/t) = \frac{-l + t(2p+k-1)}{l + t(2p+k-1)} \alpha_{p-1}(l/t). 
\end{equation}
If $\Img(l) \neq 0$, $\Gamma(X,\Prin_{\lambda,k})$ is the $(\Lieg,K)$ underlying a (nonunitary) principal series representation of $\SLR$. In this case the values of $a_p(l/t)$ are well-defined for all $t \in \real$ and they converge to constants $\alpha^{0}_p$ as $t \to 0$ and \eqref{eq:alpha_t} becomes
 \begin{equation}\label{eq:limit_alpha}
   \alpha^0_p  = -\alpha^0_{p-1}, \quad \forall p \in \ZZ. 
 \end{equation}
Hence $\alpha^0_p = (-1)^p$ no matter when $k=0$ or $1$. Moreover, $A(\lambda/t,k)$ and $A(-\lambda/t,k)$ are inverses to each other. We have proved the following result.
 
\begin{proposition}
For $\lambda=l \rho$ such that $\Img(l) \neq 0$, the intertwining operator $A(\lambda/t,k)$ from $\Gamma(X,\Prin_{-\lambda/t, k})$ to $\Gamma(X,\Prin_{\lambda/t,k})$ is an isomorphism of $(\Lieg_t,K)$-modules for $t \in \real$. The inverse of $A(\lambda/t,k)$ is $A(-\lambda/t,k)$.
\end{proposition}

Such intertwining operators fit nicely with our previous discussion of the $K$-orbits of $\Lies^*$ associated to $(\Lieg_0,K)$-modules. Recall that the equations \eqref{eq:efh0} determine an isomorphism $\Phi_\lambda$ between $Q$ and the $K$-orbit $C_\lambda$ in $\Lies^*$ determined by the equation \eqref{eq:orbit}. On the other hand, $C_\lambda = C_{-\lambda}$. Denote by $\Lambda$ the automorphism of $\Lies^*$ which sends any vector $v$ to $-v$. Then $\Lambda$ preserves $C_\lambda$  and intertwines $\Phi_{\lambda}$ with $\Phi_{-\lambda}$, as well as the line bundles over the orbits.

When $l=0$, equation \eqref{eq:alpha} becomes $\alpha_p=\alpha_{p-1}$ so $a_p=1$ for all $p \in \ZZ$. In this case we get nothing but the identity maps of $\Prin_{0,k}$ as well as the corresponding $(\Lieg_0,K)$-modules.

In general, when $l \neq 0$ is a real number or $t \in \complex$, the values of $l/t$ can take integral values infinitely many times as shown in Proposition \ref{prop:prin_int}, so that $\alpha_p(l/t)$ have zeros or poles as rational functions in $t$. At those values $A(\lambda/t,k)=A(l \rho/t, k)$ are no longer isomorphisms or not well-defined. In particular, when $l/t$ is a positive integer the map 
  \[ A(l \rho/t, k): \Gamma(X,\Prin_{-\lambda/t, k}) \to \Gamma(X,\Prin_{\lambda/t,k}) \] 
is still defined but it is of finite rank and its image in $\Gamma(X,\Prin_{\lambda/t,k})$ is the unique finite dimensional irreducible submodule as in Proposition \ref{prop:prin_int}. Nevertheless, the limit of $\alpha_p(l/t)$ as $t \to 0$ still exists as in \eqref{eq:limit_alpha} and gives rise to the intertwining operator between the irreducible motion group representations $M_0(-\lambda, Q, \phi)$ and $M_0(\lambda, Q, \phi)$ as discussed above. It does not conflict with the Mackey-Higson bijection, however, since $\Gamma(X,\Prin_{l \rho,k})$ with negative integral $l$ is reducible and so it is excluded from consideration for the Mackey-Higson bijection.

\bibliographystyle{alpha}
\bibliography{MackeySL2.bib}

\end{document}